\pdfoutput=1

\documentclass[11pt,a4paper]{amsart}

\calclayout

\usepackage{mathtools}
\usepackage{amssymb}
\usepackage{stmaryrd}
\usepackage{hyperref}
\usepackage{subcaption}

\hypersetup{bookmarksdepth=subsection}
\numberwithin{equation}{section}
\newtheorem*{mtheorem}{Main Theorem}
\newtheorem{lemma}{Lemma}[section]

\newtheorem{proposition}{Proposition}[section]

\theoremstyle{definition}
\newtheorem{definition}{Definition}[section]
\theoremstyle{remark}
\newtheorem{remark}{Remark}[section]

\newcommand{\Mod}[1]{\ (\mathrm{mod}\ #1)}

\newcommand{\sgn}{\operatorname{sgn}}
\newcommand{\defn}{\stackrel{\textrm{\scriptsize def}}{=}}

\DeclareMathOperator*{\argmin}{arg\,min}

\title{Expanding total sieve and patterns in primes}
\author{Andrzej Bo\.zek}

\email{abozek@prz.edu.pl}

\begin{document}

\begin{abstract}
Let $\big(\mathcal{S}_n^{\alpha,\kappa,\mathfrak{r}}(z)\big)_{n=1}^\infty$ be a parametrised sequence of the largest possible integer intervals, such that $z\in\mathcal{S}_n^{\alpha,\kappa,\mathfrak{r}}(z)\subset\overline{\mathcal{M}}_n^{\alpha,\kappa,\mathfrak{r}}=\bigcup_{i=1}^n [\mathfrak{r}_i]_{\mathfrak{p}_i}$ or $\mathcal{S}_n^{\alpha,\kappa,\mathfrak{r}}(z)=\emptyset$, where $z\in\mathbb{Z}$, $\mathfrak{p}_i=p_{\alpha+\left\lceil i/\kappa\right\rceil-1}$, and $p_k$ is the $k$-th prime. We prove that $\big(\#\mathcal{S}_n^{\alpha,\kappa,\mathfrak{r}}(z)\big)_{n=1}^\infty$ oscillates infinitely many times around $\beta_n\!=\!o\left(n^2\right)$ for any set of fixed $\alpha\in\mathbb{Z}^+$, $\kappa\in\mathbb{Z}\cap[1,p_\alpha)$, and $\mathfrak{r}_i\in\mathbb{Z}$. Let $\left(\mathcal{X}_n^{T,k,\rho,\eta}\right)_{n=1}^\infty$ be a parametrised sequence of sets of all the integers $x\in[\rho]_\eta$, $\rho\in\mathbb{Z}$, $\eta\in\mathbb{P}$, such that an occurrence of a given admissible $k$-tuple $T=(a_1,a_2,\ldots,a_k)$ at the position $x$ is sieved out by a set $\mathcal{M}_{n+\alpha-1}$, where $\mathcal{M}_g=\bigcup_{i=1}^g [0]_{p_i}$, in other words $\mathcal{X}_n^{T,k,\rho,\eta}=\left\{x\in[\rho]_\eta\,:\,\{x\!+\!a_1,x\!+\!a_2,\ldots,x\!+\!a_k\}\cap\mathcal{M}_{n+\alpha-1}\neq\emptyset\right\}$. We prove that for any admissible $k$-tuple $T$ and for some fixed $\alpha$, $\kappa$, $\rho$, $\eta$, $z$, and $\mathfrak{r}$, there exists a linear bijection between $\overline{\mathcal{M}}_{\kappa n}^{\alpha,\kappa,\mathfrak{r}}$ and $\mathcal{X}_n^{T,k,\rho,\eta}$ for each $n\in\mathbb{Z}^+$. It implies that the length of any expanding integer interval on which all occurrences of $T$ are sieved out by $\mathcal{M}_{n+\alpha-1}$ oscillates infinitely many times around $\widetilde{\beta}_n=o\left(n^2\right)$. According to the concept of the sieve of Eratosthenes, $\mathcal{E}_n=[2,p^2_{n+\alpha})\cap\left(\mathbb{Z}\setminus\mathcal{M}_{n+\alpha-1}\right)\subset\mathbb{P}$. Therefore, having $p^2_{n+\alpha}=\omega\left(n^2\right)$, we obtain that $\mathcal{E}_n$ includes a subset matched to $T$ for infinitely many values of $n$ and, consequently, $T$ matches infinitely many positions in the sequence of prime numbers.
\end{abstract}


\maketitle

\section{Introduction}
\label{sec:Introduction}
There exists a well-known simple necessary condition for a $k$-tuple of integers to match infinitely many positions in the sequence of prime numbers, called admissibility. On the other hand, it turns out to be very difficult to formulate and check a sufficient condition, both in a general form and in a form related to specific tuples. While it has been strongly believed that every admissible $k$-tuple matches infinitely many positions in the sequence of primes, this statement has had only a conjectural flavour for a long time. In particular, the widely-known first Hardy-Littlewood conjecture~\cite{hardy1923} predicts asymptotically growing function representing the number of prime constellations, i.e., the $k$-tuples of minimum diameters, and, as a consequence, states that this number is unbounded. The Hardy-Littlewood conjecture has been supported only by experimental results since 2013, when Zhang proved that $H_1<7\times 10^7$~\cite{zhang2014}, where
\[
	H_m=\liminf\limits_{n\to\infty}\left(p_{n+m}-p_n\right),\quad m=1,2,\ldots
\]
Shortly later, Maynard~\cite{maynard2015} and participants of Polymath project~\cite{polymath2014} improved Zhang's result, obtaining $H_1\leq 246$. Maynard proved, as well, that $H_m$ is finite for every $m$~\cite{maynard2015}, which is a strong qualitative generalization of the result related to $H_1$ and, in fact, it is the first proved statement confirming infinitude of general classes of $k$-tuples in the sequence of primes. The weakness of these breakthrough results is, however, that they involve some collections of $k$-tuples and no unambiguous statement has been formulated for a single $k$-tuple so far.

In this work, we propose an approach to prove that the admissibility is a sufficient condition for every $k$-tuple to be repeated infinitely many times in the sequence of primes. The approach is based on a concept of expanding total sieve, which is a generalized model of the sieving process of $k$-tuple matching positions. In particular, with the use of the expanding total sieve, we show that the sieving process is not efficient enough to cancel all repetitions of a given $k$-tuple inside the asymptotically growing sieve of Eratosthenes.

Standard notation is used in the work, wherein, for concise:
\[
	\left(a_n\right)\defn\left(a_n\right)_{n=1}^\infty,\qquad
	\llbracket a,b \rrbracket\defn [a,b]\cap\mathbb{Z}.
\]

The precise definitions and properties related to the concept of an expanding total sieve are given in the next section, here, we provide a brief explanation. Consider an infinite sequence of pairwise distinct residue classes $[\mathfrak{r}_n]_{\mathfrak{p}_n}$, ${n\in\mathbb{Z}^+}$, where $\mathfrak{p}_n\in\mathbb{P}$ and the sequence $\left(\mathfrak{p}_n\right)$ is non-decreasing. Let $\mathcal{M}_n^{\mathfrak{p},\mathfrak{r}}=\bigcup\limits_{i=1}^n [\mathfrak{r}_i]_{\mathfrak{p}_i}$, and let $\mathcal{P}_n^{\mathfrak{p},\mathfrak{r}}:\mathbb{Z}\mapsto\{0,1\}$ be the characteristic function of the set $\mathbb{Z}\setminus\mathcal{M}_n^{\mathfrak{p},\mathfrak{r}}$. The function $\mathcal{P}_n^{\mathfrak{p},\mathfrak{r}}$ defines a pattern on the ordered set $\mathbb{Z}$, such that $\mathcal{P}_n^{\mathfrak{p},\mathfrak{r}}(z)=0$ if $z$ is sieved out by a residue, i.e. $z\in\mathcal{M}_n^{\mathfrak{p},\mathfrak{r}}$, and $\mathcal{P}_n^{\mathfrak{p},\mathfrak{r}}(z)=1$ otherwise. As an example, suppose that $\mathfrak{p}_n=(3,3,5,5,7,7,11,11,\ldots)$ and $\mathfrak{r}_n=(1,2,4,0,5,6,7,10,\ldots)$. The resulting patterns $\mathcal{P}_n^{\mathfrak{p},\mathfrak{r}}$, for $n\in\llbracket 1,8\rrbracket$, in the domain restricted to $\llbracket 1,38\rrbracket$, are presented in Fig.~\ref{fig:examp}. Let $\mathcal{S}_n^{\mathfrak{p},\mathfrak{r}}(z)$, where $z\in\mathbb{Z}$, be the largest integer interval (in the sense of inclusion), such that $z\in\mathcal{S}_n^{\mathfrak{p},\mathfrak{r}}(z)\subset\mathcal{M}_n^{\mathfrak{p},\mathfrak{r}}$ and, in particular, $\mathcal{S}_n^{\mathfrak{p},\mathfrak{r}}(z)=\emptyset$ if $z\not\in\mathcal{M}_n^{\mathfrak{p},\mathfrak{r}}$. Hence, an interval $\mathcal{S}_n^{\mathfrak{p},\mathfrak{r}}(z)$ contains integers which are locally "totally" sieved out by residues in $\mathcal{M}_n^{\mathfrak{p},\mathfrak{r}}$ and so it will be referred to as a total sieve around $z$. For example, according to Fig.~\ref{fig:examp},we have: $\mathcal{S}_3^{\mathfrak{p},\mathfrak{r}}(23)=\llbracket 22,26\rrbracket$, $\mathcal{S}_5^{\mathfrak{p},\mathfrak{r}}(9)=\mathcal{S}_5^{\mathfrak{p},\mathfrak{r}}(12)=\mathcal{S}_5^{\mathfrak{p},\mathfrak{r}}(17)=\llbracket 7,17\rrbracket$, $\mathcal{S}_7^{\mathfrak{p},\mathfrak{r}}(21)=\emptyset$. In this work, we will not exploit the concept of a total sieve itself, but rather the construct of a sequence of total sieves $\left(\mathcal{S}_n^{\mathfrak{p},\mathfrak{r}}(z)\right)$ around a fixed integer $z$. This construct will be called an expanding total sieve, because it is easy to notice that $\mathcal{S}_n^{\mathfrak{p},\mathfrak{r}}(z)\subseteq\mathcal{S}_{n+1}^{\mathfrak{p},\mathfrak{r}}(z)$, as a consequence of inclusion $\mathcal{M}_n^{\mathfrak{p},\mathfrak{r}}\subseteq\mathcal{M}_{n+1}^{\mathfrak{p},\mathfrak{r}}$. In~Fig.~\ref{fig:examp}, the example of $\left(\mathcal{S}_n^{\mathfrak{p},\mathfrak{r}}(7)\right)$ is depicted and its 8-th element $\mathcal{S}_8^{\mathfrak{p},\mathfrak{r}}(7)=\llbracket 4,35\rrbracket$ is denoted. To be completely precise, we will use an $(\alpha,\kappa)$--regular variants of the sieving pattern and related expanding total sieve, denoted by $\mathcal{P}_n^{\alpha,\kappa,\mathfrak{r}}$ and $\left(\mathcal{S}_n^{\alpha,\kappa,\mathfrak{r}}(z)\right)$, respectively, in which $\mathfrak{p}_n=p_{\alpha+\left\lceil n/\kappa\right\rceil-1}$, where $\alpha,\kappa\in\mathbb{Z}^+$, $\kappa<p_\alpha$.

Employing the concept of the $(\alpha,\kappa)$--regular expanding total sieve, the reasoning flow used in the work can be sketched in the following points:
\begin{enumerate}
\item We prove that each sequence $\left(\#\mathcal{S}_n^{\alpha,\kappa,\mathfrak{r}}(z)\right)$ oscillates infinitely many times around $\beta_n=o(n^2)$ [Lemma~\ref{lem:totalSieveUB}, Proposition~\ref{prop:primesOm}]. The asymptotic upper bound on the expansion rate of $\left(\#\mathcal{S}_n^{\alpha,\kappa,\mathfrak{r}}(z)\right)$ is the crucial result of the work, which quite easily translates into Main Theorem.
\label{it:oscillations}
\item Let $\mathcal{P}_n$ be defined as a variant of $\mathcal{P}_n^{\mathfrak{p},\mathfrak{r}}$, such that $\mathfrak{p}_n=p_n$ and $\mathfrak{r}_n=0$ for each $n\in\mathbb{Z}^+$. Let $T=(a_1,a_2,\ldots,a_k)$, where $a_1<a_2<\ldots<a_k$, be an admissible $k$-tuple. It is an obvious consequence of admissibility, that $T$ matches $\mathcal{P}_n$ for some $n\in\mathbb{Z}^+$ at some position ${x\in\mathbb{Z}^+}$, moreover, the matching repeats in an arithmetic progression constituting the residue class $[x]_{p_n\#}$, because $\mathcal{P}_n$ is periodic, with the fundamental period $p_n\#$ [Lemma~\ref{lem:sievePatPeriod}, Remark~\ref{rem:patAvgDens}]. In the subsequent patterns $\mathcal{P}_f$, $f={n+1},{n+2},\ldots$, the set $[x]_{p_n\#}$ of matchings of $T$ is systematically reduced, because $(p_a,p_b)=1$ for $a\neq b$. Let us describe the sieving of the $k$-tuple $T$ matching set $[x]_{p_n\#}$ using a notation $^\dagger\mathcal{P}^{T,k,d,m}_g(z)$, where $m=\min\left\{[x]_{p_n\#}\cap\mathbb{Z}^+\right\}$, $d=n+1$, $g\in\mathbb{Z}^+$, $z\in\mathbb{Z}$, stating that $^\dagger\mathcal{P}^{T,k,d,m}_g(z)=1$ if $T$ matches the position $m+(z-1)p_n\#$ in $\mathcal{P}_{n+g}$, and $^\dagger\mathcal{P}^{T,k,d,m}_g(z)=0$ otherwise.
\label{it:matchset}
\item One can prove that $^\dagger\mathcal{P}^{T,k,d,m}_g$, such that $a_k-a_1<p_d$, is an $(\alpha,\kappa)$--regular sieving pattern. In other words, for each well defined $^\dagger\mathcal{P}^{T,k,d,m}_g$ there exists $\mathcal{P}_n^{\alpha,\kappa,\mathfrak{r}}$, such that $^\dagger\mathcal{P}^{T,k,d,m}_g=\mathcal{P}_n^{\alpha,\kappa,\mathfrak{r}}$, where, in particular, $\alpha=d$, $\kappa=k$, and $n=kg$ [Lemma~\ref{lem:tupleSieveing}].
\label{it:patternequiv}
\item The pattern $\mathcal{P}_n$ restricted to the interval $E_n=\left\llbracket 2,p_{n+1}^2\!-\!1\right\rrbracket$ represents the sieve of Eratosthenes constructed up to the prime $p_n$, in the sense that, if $z\in E_n$ and $\mathcal{P}_n(z)=1$, then $z\in\mathbb{P}$ [Lemma~\ref{lem:patEratPrimes}]. Hence, if an admissible $k$-tuple $T$ matches primes only finitely many times, each residue class $[m]_{p_{d-1}\#}$ of the matching points of $T$ in $\mathcal{P}_n$ is completely sieved out inside $\left\llbracket z^*,p_{n+1}^2\!-\!1\right\rrbracket$ for some fixed $z^*\in\mathbb{Z}^+$, and for each $n\in\mathbb{Z}^+$. Therefore, according to the point~(\ref{it:matchset}), if $n=d+g-1$, then $^\dagger\mathcal{P}^{T,k,d,m}_g(z)=0$ for each $z\in Z_g$ such that $Z_g\subseteq Z_{g+1}$ and $\#Z_g\asymp p_g^2=\omega\left(g^2\right)$.
\label{it:asympgrow}
\item Combining the results from the points~(\ref{it:patternequiv}) and~(\ref{it:asympgrow}), we obtain that, if an admissible $k$-tuple $T$ matches primes finitely many times, then there exists $\mathcal{P}^{d,k,\mathfrak{r}}_{kg}={^\dagger\mathcal{P}^{T,k,d,m}_g}$ and, consequently, there exists an expanding total sieve $\mathcal{S}^{d,k,\mathfrak{r}}_{kg}(\widetilde{z})$ around some $\widetilde{z}\in\mathbb{Z}$, such that $Z_g\subseteq\mathcal{S}^{d,k,\mathfrak{r}}_{kg}(\widetilde{z})$. Therefore, according to the point~(\ref{it:asympgrow}), we get $\#\mathcal{S}^{d,k,\mathfrak{r}}_{kg}(\widetilde{z})=\omega\left(g^2\right)$. However, from the point~(\ref{it:oscillations}), we have that $\left(\#\mathcal{S}^{d,k,\mathfrak{r}}_{kg}(\widetilde{z})\right)$ oscillates infinitely many times around $\beta_g=o(g^2)$, hence, we obtain a contradiction.
\end{enumerate}

\begin{figure}
	\centering
	\includegraphics[width=\textwidth]{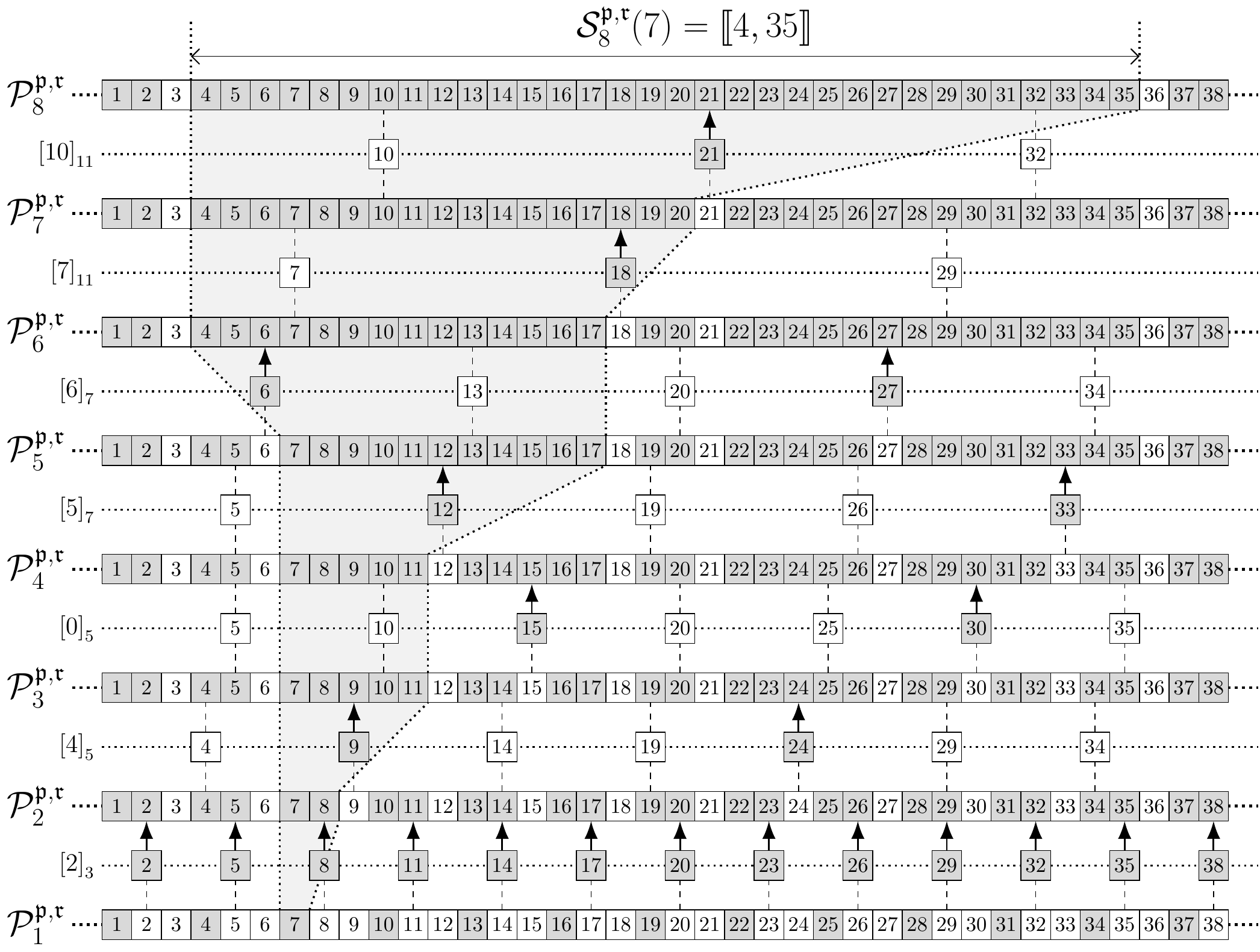}
	\caption{An example of an expanding total sieve around $z=7$}
	\label{fig:examp}
\end{figure}

In the work, we use loose estimates of the asymptotic grow rates $p_n=\omega(n)$ [Lemma~\ref{lem:primesOmN}] and $\prod_{i=0}^{n-1}\frac{p_{\alpha+i}}{p_{\alpha+i}-\kappa}=o(n)$ [Lemma~\ref{lem:patAvgLenAsymp}], introduced \textit{ad hoc} by simple proofs. Therefore, the result presented here does not rely directly on the prime number theorem~\cite{poussin1896,hadamard1896} and the generalized form of the third Mertens theorem~\cite{mertens1874} which provide well-known much more exact estimates for the mentioned sequences.

\subsection{Guiding illustrative example}
\label{subsec:guidexamp}

To make the proposed approach easier to follow, we provide an illustrative example based on the arbitrarily chosen admissible triplet $(0,2,6)$. Hence, we consider the $k$-tuple such that $k=3$ and $T=(a_1,a_2,a_3)=(0,2,6)$.

The approach can be split into three conceptual phases. In the first phase, we construct an initial setting, sieving $\mathbb{Z}$ by residue classes $[0]_{p_i}$ for $i\in\llbracket 1,d\!-\!1\rrbracket$, and choosing an arithmetic progression of matching positions of $T$ in the sieved set. In the second phase, we consider an infinite process of sieving by all the subsequent residue classes $[0]_{p_d},[0]_{p_{d+1}},\ldots$ in the context of cancellation of the $k$-tuples from the arithmetic progression constructed in the previous phase. The value of $d$ has to be large enough for the method to work correctly, we choose $d=4$. In the third phase, we employ the concept of the basic sieve of Eratosthenes, along with the original idea of the expanding total sieve, to show that the modelled sieving process is not efficient enough to eliminate all but finite number of matchings of $T$ in the sequence of primes.

\begin{figure}
	\centering
	\includegraphics[width=1.42\textwidth,angle=90]{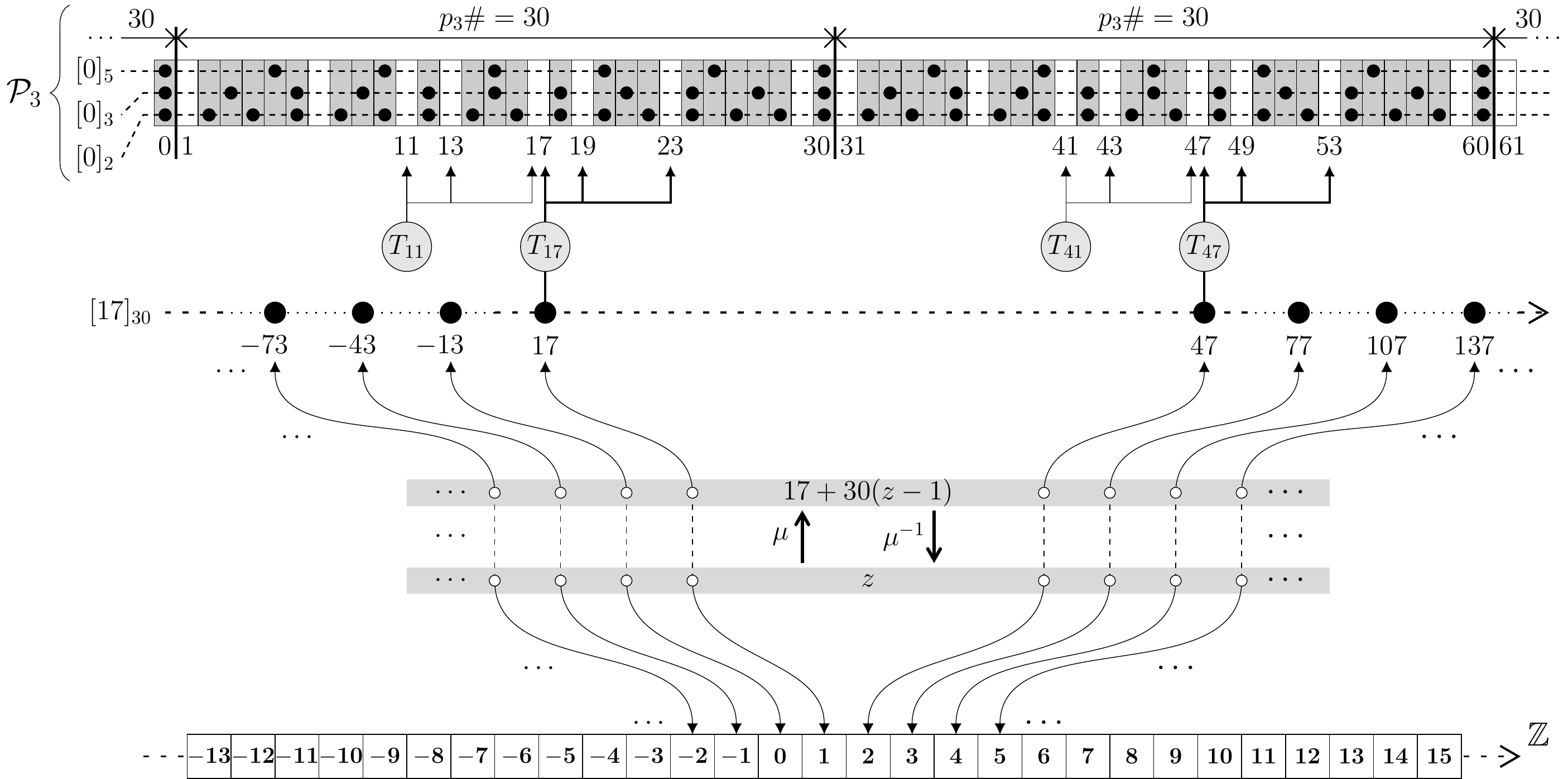}
	\caption{An example based on the triplet $(0,2,6)$, the initial setting}
	\label{fig:illustr1}
\end{figure}
\begin{figure}
	\centering
	\includegraphics[width=1.42\textwidth,angle=90]{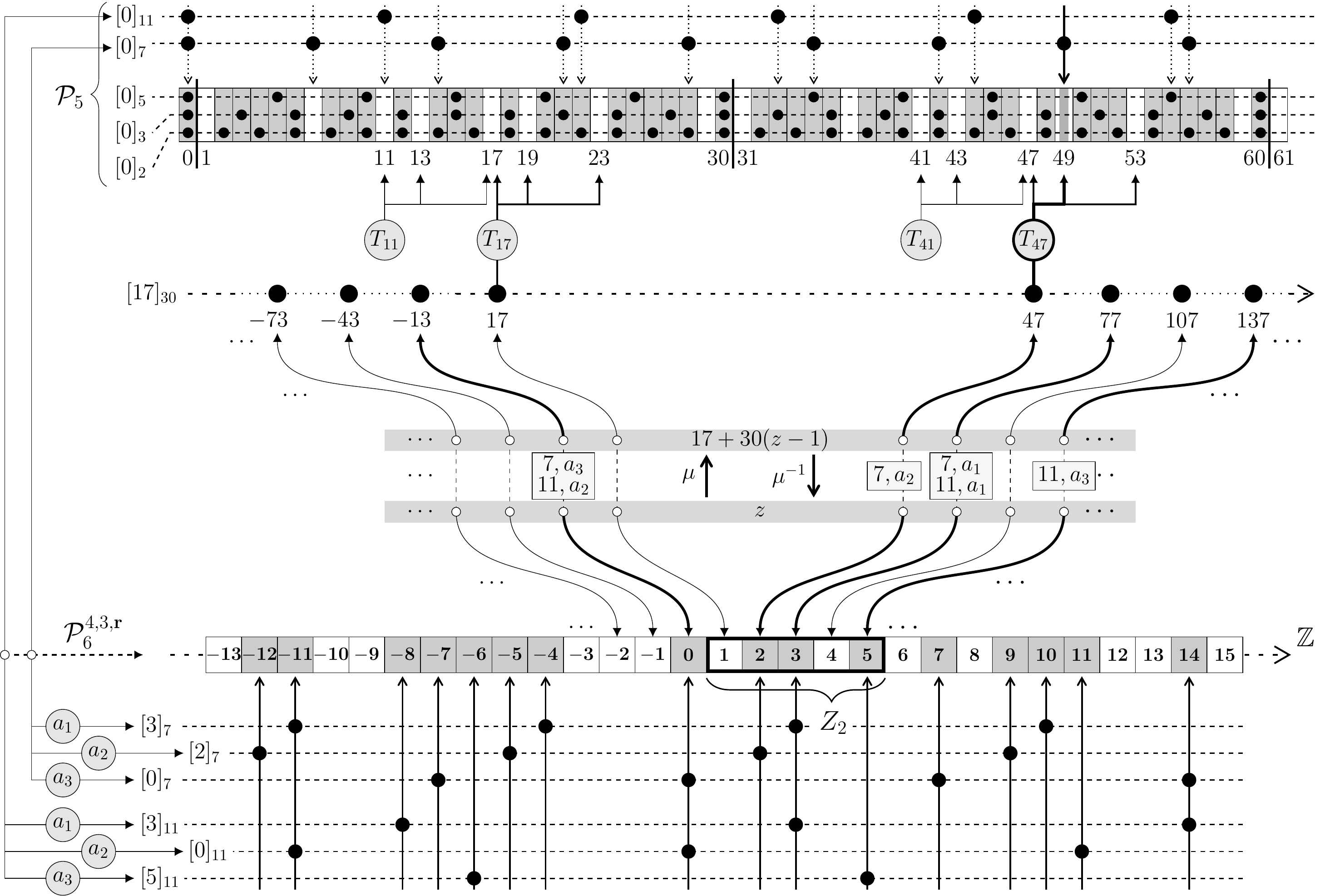}
	\caption{An example based on the triplet $(0,2,6)$, the sieving procedure}
	\label{fig:illustr2}
\end{figure}

The first phase is illustrated in Fig.~\ref{fig:illustr1}. At the top of the figure, the result of sieving $\mathbb{Z}$ by residue classes $[0]_{p_1},[0]_{p_2},[0]_{p_3}$ is presented, and it is described by the function $\mathcal{P}_3$, where $\mathcal{P}_n$ is defined as the characteristic function of the set $\mathbb{Z}\setminus \bigcup_{i=1}^n [0]_{p_i}$. In general, every function $\mathcal{P}:\mathbb{Z}\to\{0,1\}$ is called a pattern in this work, and we will say that a $k$-tuple $T=(a_1,a_2,\ldots,a_k)$ matches a patter $\mathcal{P}$ at a position $m$, or that there exists a matched instance $T_m$, if and only if $\mathcal{P}(m+a_1)=\mathcal{P}(m+a_2)=\ldots=\mathcal{P}(m+a_k)=1$. According to Fig.~\ref{fig:illustr1}, in the considered example, we have matched instances $T_{11}$, $T_{17}$, $T_{41}$, and $T_{47}$ inside $\llbracket 0,61\rrbracket$. It is well known that the pattern $\mathcal{P}_n$ repeats modulo $p_n\#$, hence, the matching positions form two residue classes: $[11]_{30}$ and $[17]_{30}$. In the proposed method, we need to choose one of these classes. Let us select $[17]_{30}$. Additionally, we define a linear bijection $\mu:\mathbb{Z}\to [17]_{30}$, such that $\mu(z)=17+30(z-1)$, which is also presented in Fig.~\ref{fig:illustr1}.

In the second phase, we analyse the sieving of $T$ instances matched at the previously selected positions in $[17]_{30}$ by residues from successively added classes $[0]_{p_i}$, $i=4,5,\ldots$ We classify a matching position $m$ as sieved out if there exists a residue $\rho\in\mathbb{Z}$ such that $\rho=m+a$ for some $a\in\{a_1,a_2,a_3\}=\{0,2,6\}$, so that $T$ is no more matched at the position $m$, after sieving by a residue class containing $\rho$. W can see, directly from Fig.~\ref{fig:illustr2}, that $49\in [0]_7$ sieves out the element $a_2$ of $T_{47}$ and, consequently, it sieves out the matched instance $T_{47}$ itself. Moreover $(7,5\#)=1$, hence, $[0]_7$ sieves out the element $a_2$ in every 7-th repetition of the matched instance $T_m$ for $m\in [17]_{30}$, thus, it sieves out all the instances $T_m$ such that $m\in [47]_{30\times 7}$. Notice that $[47]_{30\times 7}$ is mapped onto $[2]_7$ by $\mu^{-1}$. Analogously, $[0]_7$ sieves out the instances matched at positions in $[77]_{30\times 7}$ (being congruent to $a_1$) and $[197]_{30\times 7}$ (being congruent to $a_3$). The residue classes $[77]_{30\times 7}$ and $[197]_{30\times 7}$ are mapped onto $[3]_7$ and $[0]_7$ by $\mu^{-1}$, respectively. The same rule can be applied to all the subsequent residue classes $[0]_{11},[0]_{13},\ldots$, the case of $[0]_{11}$ is included in Fig.~\ref{fig:illustr2}. Finally, we observe that the set of matching positions of instances of $T=(0,2,6)$ in $[17]_{30}$, after sieving these instances by $\bigcup_{i=4}^{3+n} [0]_{p_i}$, and after mapping by $\mu^{-1}$, results in
\begin{align*}
	\mathcal{I}_{3n}^{4,3,\mathbf{r}}%
	=\mathbb{Z}\,\setminus\,\big(&[3]_7\cup [2]_7\cup [0]_7\cup [3]_{11}\cup [0]_{11}\cup [5]_{11}\\
	&\cup [\mathbf{r}_7]_{13}\cup [\mathbf{r}_8]_{13}\cup [\mathbf{r}_9]_{13}\cup\cdots
	\cup [\mathbf{r}_{3n-2}]_{p_{3+n}}\cup [\mathbf{r}_{3n-1}]_{p_{3+n}}\cup [\mathbf{r}_{3n}]_{p_{3+n}}\big)
\end{align*}
for some sequence $\left(\mathbf{r}_i\right)_{i=1}^{3n}$. Let $\mathcal{P}_{3n}^{4,3,\mathbf{r}}$ be the characteristic function of $\mathcal{I}_{3n}^{4,3,\mathbf{r}}$.

In the third phase, we start from the simple observation, that the pattern $\mathcal{P}_5$ (see Fig.~\ref{fig:illustr2}), in the domain restricted to $\llbracket 2,168\rrbracket$, represents the basic sieve of Eratosthenes, in the sense that, if $z\in\llbracket 2,168\rrbracket$ and $\mathcal{P}_5(z)=1$, then $z\in\mathbb{P}$. The triplets $(m,m+2,m+6)$, $m\in [17]_{30}$, are entirely included in $\llbracket 2,168\rrbracket$ for $m\in M_2=\{17,47,77,107,137\}$. The set $M_2$ is mapped onto $Z_2=\llbracket 1,5\rrbracket$ by $\mu^{-1}$ (Fig.~\ref{fig:illustr2}). Therefore, the triplet $\big(\mu(z),{\mu(z)+2},{\mu(z)+6}\big)$, $z\in Z_2$, matches primes if and only if $\mathcal{P}_6^{4,3,\mathbf{r}}(z)=1$. This observation can be extended to any $n\in\mathbb{Z}^+$, and the related $\mathcal{P}_{d-1+n}$, $\mathcal{P}_{kn}^{d,k,\mathbf{r}}$, $M_n$, $Z_n$, while the case for $d=4$, $k=3$, $n=2$ is presented in Fig.~\ref{fig:illustr2}. In general, the pattern $\mathcal{P}_{d-1+n}$ includes the representation of the sieve of Eratosthenes in the subset of its domain $E_n=\llbracket 2,p_{d+n}^2\!-\!1\rrbracket$. In the asymptotic behaviour under $n\to\infty$, we have $\#E_n\asymp\#M_n=\#Z_n$, because the distance between consecutive ordered elements of $M_n$ is equal and it does not depend on $n$. Therefore, from the obvious property $\#E_n\sim p_n^2=\omega(n^2)$, we obtain $\#Z_n=\omega(n^2)$. Suppose that $T=(0,2,6)$ matches primes only finitely many times. Hence, there exists the largest $m^*\in[17]_{30}$, such that $\{m^*,m^*\!+\!2,m^*\!+\!6\}\subset\mathbb{P}$. Consequently, $\mathcal{P}_{3n}^{4,3,\mathbf{r}}(z)=0$ for each $z\in Z^*_n=\big(\mu^{-1}(m^*),\infty\big)\cap Z_n$. We get immediately that $\#Z^*_n=\omega(n^2)$. We obtain, as well, that $Z^*_n\subseteq \mathcal{S}_n^{\alpha,\kappa,\mathfrak{r}}(z)$ for each $n\in\mathbb{Z}^+$, where the sequence $\left(\mathcal{S}_n^{\alpha,\kappa,\mathfrak{r}}(z)\right)$ conforms the definition of an $(\alpha,\kappa)$--regular expanding total sieve under $\alpha=d=4$, $\kappa=k=3$, $\mathfrak{r}=\mathbf{r}$, and some properly fixed $z$. However, we prove in Lemma~\ref{lem:totalSieveUB} that the sequence $\left(\#\mathcal{S}_n^{\alpha,\kappa,\mathfrak{r}}(z)\right)$ does not grow systematically faster than $o(n^2)$, hence, we obtain a contradiction.

\section{Expanding total sieve}
\label{sec:sieve}
\begin{definition}
A sequence $\mathcal{M}^{\mathfrak{p},\mathfrak{r}}:\mathbb{Z}^+\mapsto 2^\mathbb{Z}$ such that $\mathcal{M}_n^{\mathfrak{p},\mathfrak{r}}=\bigcup\limits_{i=1}^n [\mathfrak{r}_i]_{\mathfrak{p}_i}$ will be referred to as an ordered sieving model, induced by a prime sieving sequence $\mathfrak{p}:\mathbb{Z}^+\mapsto\mathbb{P}$ and a residue sieving sequence $\mathfrak{r}:\mathbb{Z}^+\mapsto\mathbb{Z}^+\cup\{0\}$, providing that
\begin{equation}
	\left(\forall n\in\mathbb{Z}^+\right)\quad\mathfrak{p}_n\leq\mathfrak{p}_{n+1},
	\label{eq:sievSequence1}
\end{equation}
\begin{equation}
	\left(\forall n\in\mathbb{Z}^+\right)\quad\mathfrak{r}_n<\mathfrak{p}_n,
	\label{eq:sievSequence2}
\end{equation}
\begin{equation}
	\left(\forall p\in\mathbb{P}\right)\quad\#\left\{n\in\mathbb{Z}^+:\mathfrak{p}_n=p\right\}<p,
	\label{eq:sievSequence3}
\end{equation}
\begin{equation}
	\left(\forall (n,m)\in\left(\mathbb{Z}^+\right)^2\right)\quad\mathfrak{p}_n=\mathfrak{p}_m
	\Longrightarrow\mathfrak{r}_n\neq	\mathfrak{r}_m.
	\label{eq:sievSequence4}
\end{equation}
\label{def:sievSequence}
\end{definition}
\begin{remark}
The model introduced in Definition~\ref{def:sievSequence} is ordered in the sense determined by~(\ref{eq:sievSequence1}). By the constraints~(\ref{eq:sievSequence2})--(\ref{eq:sievSequence4}), we exclude in advance the trivial cases, such that $\mathcal{M}_n^{\mathfrak{p},\mathfrak{r}}=\mathbb{Z}$ for all $n\geq N\in\mathbb{Z}^+$, or $\mathcal{M}_n^{\mathfrak{p},\mathfrak{r}}=\mathcal{M}_{n+1}^{\mathfrak{p},\mathfrak{r}}$ for some $n\in\mathbb{Z}^+$.
\end{remark}

\begin{definition}
A characteristic function of $\mathbb{Z}\setminus\mathcal{M}_n^{\mathfrak{p},\mathfrak{r}}$ under a fixed value of $n$, i.e., the function $\mathcal{P}_n^{\mathfrak{p},\mathfrak{r}}:\mathbb{Z}\mapsto\{0,1\}$, such that
\[
	\mathcal{P}_n^{\mathfrak{p},\mathfrak{r}}(z)=\begin{cases}
		0 &\text{if }z\in\mathcal{M}_n^{\mathfrak{p},\mathfrak{r}},\\
		1 &\text{otherwise},
	\end{cases}
\]
will be referred to as a sieving pattern induced by $\mathcal{M}^{\mathfrak{p},\mathfrak{r}}_n$.
\label{def:sievPattern}
\end{definition}

\begin{definition}
A function $\mathcal{S}_n^{\mathfrak{p},\mathfrak{r}}:\mathbb{Z}\mapsto 2^\mathbb{Z}$, such that
\[
\begin{split}
	\mathcal{S}_n^{\mathfrak{p},\mathfrak{r}}(z)=\Big\{x\in\llbracket a\!+\!1,b\!-\!1\rrbracket:\;\;
	(a,b)\in\mathbb{Z}^2,\;&a\leq z\leq b,\;\sgn(z\!-\!a)=\sgn(b\!-\!z),\\
	&\mathcal{P}_n^{\mathfrak{p},\mathfrak{r}}(a)=\mathcal{P}_n^{\mathfrak{p},\mathfrak{r}}(b)=1,\;
	\sum\limits_{i=a}^b \mathcal{P}_n^{\mathfrak{p},\mathfrak{r}}(i) \leq 2\Big\},
\end{split}
\]
will be called a total sieving function, and the set $\mathcal{S}_n^{\mathfrak{p},\mathfrak{r}}(z)$ will be referred to as a total sieve around $z$ in the sieving pattern $\mathcal{P}_n^{\mathfrak{p},\mathfrak{r}}$.
\end{definition}

\begin{definition}
A sequence of sets $\left(\mathcal{S}_n^{\mathfrak{p},\mathfrak{r}}(z)\right)$ under fixed value of $z$ will be referred to as an expanding total sieve around $z$.
\end{definition}

\begin{lemma}
A sieving pattern $\mathcal{P}_n^{\mathfrak{p},\mathfrak{r}}$ is a periodic function and it has the fundamental period
\begin{equation}
	\mathcal{T}_n^\mathfrak{p}=
	\prod_{k\in\mathbb{Z}^+\::\:\left(\exists\,z\in \llbracket 1,n\rrbracket\right)\;\mathfrak{p}_z=p_k}p_k.
\label{eq:sievPatPeriod}
\end{equation}
\label{lem:sievePatPeriod}
\end{lemma}
\begin{proof}
Notice that, for $\eta=\mathfrak{p}_1=\mathfrak{p}_2=\ldots=\mathfrak{p}_n<\mathfrak{p}_{n+1}$, we have $\mathcal{T}_n^\mathfrak{p}=\eta$, because residues in all the involved classes repeat modulo $\eta$. Observe that the obtained value conforms the formula~(\ref{eq:sievPatPeriod}). Suppose that a function $\mathcal{P}_n^{\mathfrak{p},\mathfrak{r}}$ has the fundamental period given by~(\ref{eq:sievPatPeriod}). Assume that $\mathfrak{p}_n<\rho=\mathfrak{p}_{n+1}=\mathfrak{p}_{n+2}=\ldots=\mathfrak{p}_{n+g}<\mathfrak{p}_{n+g+1}$. According to Definition~\ref{def:sievSequence}, we have $\left(\rho,\mathcal{T}_n^\mathfrak{p}\right)=1$, hence, the same value of residues modulo $\rho$ repeats exactly every $\rho$-th multiply of $\mathcal{T}_n^\mathfrak{p}$, thus, $\mathcal{T}_{n+g}^\mathfrak{p}=\rho\mathcal{T}_n^\mathfrak{p}$, which is consistent with~(\ref{eq:sievPatPeriod}). We obtain, therefore, by induction, that the formula~(\ref{eq:sievPatPeriod}) is correct.
\end{proof}

\begin{lemma}
A sieving pattern $\mathcal{P}_n^{\mathfrak{p},\mathfrak{r}}$ has the average value
\begin{equation}
	\mathcal{D}_n^\mathfrak{p}=
	\prod\limits_{z=1}^{\infty}\frac{p_z-\#\{i\in \mathbb{Z}^+ : i\leq n, \mathfrak{p}_i=p_z\}}{p_z}.
\label{eq:patAvgDens}
\end{equation}
\label{lem:patAvgDens}
\end{lemma}
\begin{proof}
By Lemma~\ref{eq:sievPatPeriod}, a sieving pattern $\mathcal{P}_n^{\mathfrak{p},\mathfrak{r}}$ is a periodic function, hence, it has a well-defined average value
\[
	\mathcal{D}_n^\mathfrak{p}=\frac{1}{\mathcal{T}_n^\mathfrak{p}}
	\sum\limits_{i=1}^{\mathcal{T}_n^\mathfrak{p}}\mathcal{P}_n^{\mathfrak{p},\mathfrak{r}}(i).
\]
We will prove that $\mathcal{D}_n^\mathfrak{p}$ is given by~(\ref{eq:patAvgDens}), using an inductive approach, similarly as for Lemma~\ref{lem:sievePatPeriod}.

Notice that, for $\eta=\mathfrak{p}_1=\mathfrak{p}_2=\ldots=\mathfrak{p}_n<\mathfrak{p}_{n+1}$, we have $\mathcal{D}_n^\mathfrak{p}=(\eta-n)/\eta$, which conforms the formula~(\ref{eq:patAvgDens}). Indeed, $\mathcal{T}_n^\mathfrak{p}=\eta$, by Lemma~\ref{lem:sievePatPeriod}, and $\sum_{k=1}^\eta\mathcal{P}_n^{\mathfrak{p},\mathfrak{r}}(k)=\eta-n$, because $\mathcal{P}_n^{\mathfrak{p},\mathfrak{r}}(k)=0$ if and only if $k\in\bigcup_{i=1}^n [\mathfrak{r}_i]_{\mathfrak{p}_i}$, according to Definitions~\ref{def:sievSequence} and~\ref{def:sievPattern}. Suppose that the average value of the function $\mathcal{P}_n^{\mathfrak{p},\mathfrak{r}}$ is given by~(\ref{eq:patAvgDens}). Assume that $\mathfrak{p}_n<\rho=\mathfrak{p}_{n+1}=\mathfrak{p}_{n+2}=\ldots=\mathfrak{p}_{n+g}<\mathfrak{p}_{n+g+1}$. According to~(\ref{eq:sievPatPeriod}), we have $\left(\rho,\mathcal{T}_n^\mathfrak{p}\right)=1$, hence, for each set $K_x=[x]_{\mathcal{T}_n^\mathfrak{p}}$, $x\in\mathbb{Z}$, such that $\mathcal{M}_n^{\mathfrak{p},\mathfrak{r}}\cap K_x=\emptyset$, there exist subsets $K_x \supset K_x^z=\left[\mathfrak{r}_z\right]_{\rho\mathcal{T}_n^\mathfrak{p}}$, where $z\in Z=\llbracket n\!+\!1,n\!+\!g\rrbracket$, such that $K_x^z\subset\mathcal{M}_{n+g}^{\mathfrak{p},\mathfrak{r}}$, according to Definitions~\ref{def:sievSequence} and~\ref{def:sievPattern}. Notice that $\#\big(K_x \cap T\big)/\#\big(K_x^z\cap T\big)=\rho$, for any $z\in Z$, $T=\left\llbracket m\!+\!1,m\!+\!\rho\mathcal{T}_n^\mathfrak{p}\right\rrbracket$, $m\in\mathbb{Z}$, where $\rho\mathcal{T}_n^\mathfrak{p}$ is the fundamental period of $\mathcal{P}_{n+g}^{\mathfrak{p},\mathfrak{r}}$, according to Lemma~\ref{lem:sievePatPeriod}. Consequently, we obtain $\mathcal{D}_{n+g}^{\mathfrak{p},\mathfrak{r}}=\frac{\rho-g}{\rho}\mathcal{D}_n^{\mathfrak{p},\mathfrak{r}}$, which conforms~(\ref{eq:patAvgDens}). Therefore, by induction, the formula~(\ref{eq:patAvgDens}) is correct.
\end{proof}

\begin{definition}
A prime sieving sequence of the form $\mathfrak{p}_n=p_{\alpha+\left\lceil n/\kappa\right\rceil-1}$, where $\alpha,\kappa\in\mathbb{Z}^+$, will be referred to as an $(\alpha,\kappa)$--regular prime sieving sequence.
\label{def:regularPrimeSeq}
\end{definition}

\begin{remark}
An ordered sieving model, sieving pattern, total sieving function, and expanding total sieve induced by an $(\alpha,\kappa)$--regular prime sieving sequence will be called $(\alpha,\kappa)$--regular and will be denoted by $\mathcal{M}^{\alpha,\kappa,\mathfrak{r}}$, $\mathcal{P}_n^{\alpha,\kappa,\mathfrak{r}}$, $\mathcal{S}_n^{\alpha,\kappa,\mathfrak{r}}$, and $\left(\mathcal{S}_n^{\alpha,\kappa,\mathfrak{r}}(z)\right)$, respectively.
\label{rem:regPatDef}
\end{remark}

\begin{remark}
It is easy to notice that, for an $(\alpha,\kappa)$--regular sieving pattern, the formulas~(\ref{eq:sievPatPeriod}) and~(\ref{eq:patAvgDens}) take the following more concise forms
\begin{equation}
	\mathcal{T}_n^{\alpha,\kappa}=\prod\limits_{i=0}^{\left\lceil n/\kappa\right\rceil-1}p_{\alpha+i}
	=\frac{p_{\alpha+\left\lceil n/\kappa\right\rceil-1}\#}{p_{\alpha-1}\#},
\label{eq:patPeriodReg}
\end{equation}
\begin{equation}
	\mathcal{D}_n^{\alpha,\kappa}=\left(1-\frac{n-\kappa\left\lfloor n/\kappa\right\rfloor}
	{p_{\alpha+\left\lfloor n/\kappa\right\rfloor}}\right)
	\prod\limits_{i=0}^{{\left\lfloor n/\kappa\right\rfloor}-1}\left(1-\frac{\kappa}{p_{\alpha+i}}\right),
\label{eq:patAvgDensReg}
\end{equation}
respectively, in particular, if $\kappa \mid n$, then
\begin{equation}
	\mathcal{D}_n^{\alpha,\kappa}=\prod\limits_{i=0}^{n/\kappa-1}\left(1-\frac{\kappa}{p_{\alpha+i}}\right).
\label{eq:patAvgDensReg2}
\end{equation}
\label{rem:patAvgDens}
\end{remark}

\begin{definition}
An $(\alpha,\kappa)$--regular sieving pattern such that $\mathfrak{p}_n=p_n$ and $\mathfrak{r}_n=0$ for each $n\in\mathbb{Z}^+$ will be called an Eratosthenes sieving pattern and will be denoted by $\mathcal{P}_n$.
\label{def:patEratosthenes}
\end{definition}

\begin{lemma}
Let $n\in\mathbb{Z}^+$. If $z\in\left\llbracket2,p_{n+1}^2\!-\!1\right\rrbracket$ and $\mathcal{P}_n(z)=1$, then $z\in\mathbb{P}$.
\label{lem:patEratPrimes}
\end{lemma}
\begin{proof}
The lemma follows from the obvious observation that the function $\mathcal{P}_n$, according to its Definition~\ref{def:patEratosthenes}, is equivalent to the sieve of Eratosthenes constructed up to the prime $p_n$, where a position $z$ is sieved or unsieved if $\mathcal{P}_n(z)=0$ or $\mathcal{P}_n(z)=1$, respectively. More explicitly, according to Definition~\ref{def:patEratosthenes}, if $\mathcal{P}_n(z)=1$, then $(p_i,z)=1$ for each $i\in\llbracket 1,n\rrbracket$, but $z<p_{n+1}^2$, therefore $z$ is a prime number.
\end{proof}

\begin{lemma}
For the sequence of prime numbers $\left(p_n\right)$, we have $p_n=\omega(n)$.
\label{lem:primesOmN}
\end{lemma}
\begin{proof}
Let $p_n=c_nn$. Suppose that the lemma is false, i.e., there exists $M\in\mathbb{R}$ such that
\begin{equation}
	\left(\forall n\in\mathbb{Z}^+\right)\quad c_n\leq M.
\label{eq:primesGRBound}
\end{equation}
According to (\ref{eq:patAvgDensReg2}),
\begin{equation}
	\mathcal{L}_n=\left(\mathcal{D}_n^{1,1}\right)^{-1}=\prod\limits_{i=1}^n\frac{p_n}{p_n-1}
	=\prod\limits_{i=1}^n\frac{c_nn}{c_nn-1}
\label{eq:primesGapLB}
\end{equation}
represents the average distance between adjacent non-residual positions in $\mathcal{P}_n$ and, consequently, the lower bound on the asymptotic average value of a prime gap
\[
	\overline{g}_n=\lim\limits_{k\to\infty}\frac{1}{k}\sum\limits_{i=1}^k g_{n+i}
\]
where $g_n=p_{n+1}-p_n$. Combining~(\ref{eq:primesGRBound}) and~(\ref{eq:primesGapLB}), we obtain
\[
	\overline{g}_n\geq\prod\limits_{i=1}^n\frac{c_nn}{c_nn-1}\geq\mathcal{L}_n^*=\prod\limits_{i=1}^n\frac{Mn}{Mn-1}.
\]
Notice that
\[
	\int\limits_1^\infty \ln\frac{Mx}{Mx-1}dx
	=\lim\limits_{t\to\infty}\left[x\ln\frac{Mx}{Mx-1}+\frac{1}{M}\ln\left(Mx-1\right)\right]_1^t=\infty,
\]
therefore, $\lim_{n\to\infty}\ln\mathcal{L}_n^*=\infty$ and $\left(\overline{g}_n\right)$ is unbounded. However, the condition~(\ref{eq:primesGRBound}) implies that $g_n\leq M$, thus we obtain a contradiction.
\end{proof}

\begin{lemma}
Let $\mathcal{L}_n^{\alpha,\kappa}=\left(\mathcal{D}_n^{\alpha,\kappa}\right)^{-1}$. We have $\mathcal{L}_n^{\alpha,\kappa}=o(n)$.
\label{lem:patAvgLenAsymp}
\end{lemma}
\begin{proof}
According to Lemma~\ref{lem:patAvgDens} and Remark~\ref{rem:patAvgDens}, we obtain
\[
	\mathcal{L}_{\kappa n}^{\alpha,\kappa}=\prod\limits_{i=0}^{n-1}\frac{p_{\alpha+i}}{p_{\alpha+i}-\kappa}
	=\prod\limits_{i=0}^{n-1}a_i.
\]
Consider a sequence $\left(b_i\right)$, such that $b_i=(i+2)/(i+1)$. It is obvious that $\prod_{i=0}^{n-1} b_i=n+1$. Notice that $a_i<b_i$ for any sufficiently large $i\in\mathbb{Z}^+$, as $p_n=\omega(n)$, according to Lemma~\ref{lem:primesOmN}. Therefore, $\mathcal{L}_{\kappa n}^{\alpha,\kappa}=o\left(\prod_{i=0}^{n-1} b_i\right)=o(n)$. From the formula~(\ref{eq:patAvgDensReg}) we have immediately that $\mathcal{L}_{n+1}^{\alpha,\kappa}>\mathcal{L}_n^{\alpha,\kappa}$ for each $n\in\mathbb{Z}^+$ and, consequently, $\mathcal{L}_n^{\alpha,\kappa}=o(n)$.
\end{proof}

\begin{lemma}
A sequence $\left(\#\mathcal{S}_n^{\alpha,\kappa,\mathfrak{r}}(z)\right)$ oscillates infinitely many times around $\beta_n$, where
\begin{equation}
	\beta_n\leq\gamma_n=\frac{2\,n\,p_{\alpha+\left\lfloor n/\kappa\right\rfloor}}
	{p_{\alpha+\left\lfloor n/\kappa\right\rfloor}-n+\kappa\left\lfloor n/\kappa\right\rfloor}
	\prod\limits_{i=0}^{{\left\lfloor n/\kappa\right\rfloor}-1}\frac{p_{\alpha+i}}{p_{\alpha+i}-\kappa},
\label{eq:totalSieveUB}
\end{equation}
for any residue sieving sequence $\mathfrak{r}$ and any $z\in\mathbb{Z}$.
\label{lem:totalSieveUB}
\end{lemma}
\begin{proof}
For concise, the variable $n$ will be called a step of expansion, and $z\in\mathbb{Z}$ such that $\mathcal{P}_n^{\alpha,\kappa,\mathfrak{r}}(z)=0$ or $\mathcal{P}_n^{\alpha,\kappa,\mathfrak{r}}(z)=1$ will be referred to as a sieved or unsieved position, respectively.

Suppose that $\mathfrak{r^*}$ is a variant of the residue sieving sequence  $\mathfrak{r}$ which maximises asymptotic grow rate of $\left(\#\mathcal{S}_n^{\alpha,\kappa,\mathfrak{r}}(z)\right)$ for fixed values of $\alpha$, $\kappa$, and $z$. In Fig.~\ref{fig:sieve}, an approximate density distribution of sieved positions inside and around the total sieve $\left(\mathcal{S}_n^{\alpha,\kappa,\mathfrak{r^*}}(z)\right)$ in the $n$-th step of expansion is represented by $\widetilde{f}_n:\mathbb{R}\mapsto[0,1]$, obtained by interpolation from $\mathbb{Z}$ to $\mathbb{R}$ and smoothing of $f_n(z)=1-\mathcal{P}_n^{\alpha,\kappa,\mathfrak{r^*}}(z)$. The density of sieved positions averages to $\mathcal{C}_n^{\alpha,\kappa}=1-\mathcal{D}_n^{\alpha,\kappa}$ on finite intervals, according to Lemmas~\ref{lem:sievePatPeriod} and~\ref{lem:patAvgDens}, therefore, the sieve $\mathcal{S}_n^{\alpha,\kappa,\mathfrak{r^*}}(z)=\left\llbracket s_n^\mathrm{L},s_n^\mathrm{R}\right\rrbracket$ is surrounded by intervals $\left\llbracket b_n^\mathrm{L},s_n^\mathrm{L}\!-\!1\right\rrbracket$ and $\left\llbracket s_n^\mathrm{R}\!+\!1,b_n^\mathrm{R}\right\rrbracket$ on which the average density of sieved positions is less than $\mathcal{C}_n^{\alpha,\kappa}$, so that the density averages to $\big(1\pm o(1)\big)\mathcal{C}_n^{\alpha,\kappa}$ on $\left\llbracket b_n^\mathrm{L},b_n^\mathrm{R}\right\rrbracket$. The distribution of sieved positions in $\mathcal{P}_n^{\alpha,\kappa,\mathfrak{r^*}}$ can be modelled more precisely as an additive superposition of contributions from all the steps $1,2,\ldots,n$. In the $i$-th step, the sieved positions included in the set $S_i=\mathcal{M}_i^{\alpha,\kappa,\mathfrak{r}^*}\setminus\mathcal{M}_{i-1}^{\alpha,\kappa,\mathfrak{r}^*}$ are added. We will approximate the characteristic function of $S_i\subset\mathbb{Z}$ by a continuous density function $\widetilde{d}_i:\mathbb{R}\mapsto[0,1]$, such that the superposition can be expressed in the form
\begin{equation}
	\widetilde{f}_n(\zeta)=\sum\limits_{i=1}^n\widetilde{d}_i(\zeta).
\label{eq:superposition}
\end{equation}
The superposition is conceptually presented in Fig.~\ref{fig:sieve}. Any function $\widetilde{d}_i$ is periodic, because the characteristic functions of $\mathcal{M}_i^{\alpha,\kappa,\mathfrak{r}^*}$ and, consequently, of $S_i$ are periodic. Hence, $\widetilde{d}_i$ oscillates infinitely many times around its mean value and it averages to this mean value on finite intervals of $\mathbb{R}$, not longer than $\mathcal{T}_i^{\alpha,\kappa}$, i.e., the fundamental period of $\widetilde{d}_i$. Observe that, to maximise $\#\mathcal{S}_n^{\alpha,\kappa,\mathfrak{r^*}}(z)$ for a given $n$, the values of $\left(\mathfrak{r^*}_i\right)_{i=1}^n$ should induce functions $\widetilde{d}_i$, such that the maxima of their oscillations are cumulated inside the sieve and, consequently, the minima are located just outside the sieve (at least in the average sense), as illustrated in Fig.~\ref{fig:sieve}. Moreover, if $i_1<i_2$, then $\widetilde{d}_{i_1}$ has greater local impact on $\widetilde{f}_n$ than $\widetilde{d}_{i_2}$, because $S_{i_1}$ has greater density in $\mathbb{Z}$ than $S_{i_2}$. Hence, we obtain that under $n\to\infty$, the closer to the sieve, the less superposed value of $\widetilde{f}_n$, with accuracy to minor fluctuations. It confirms the shape of the function $\widetilde{f}_n$ presented in Fig.~\ref{fig:sieve}.

Therefore, in the asymptotic behaviour, on some intervals adjacent to $\mathcal{S}_n^{\alpha,\kappa,\mathfrak{r^*}}(z)$ from the outside, the average density of sieved positions is stabilised at the value below $\mathcal{C}_n^{\alpha,\kappa}$ or, equivalently, the average distance $\widetilde{\mathcal{L}}_n^{\alpha,\kappa}(z)$ between adjacent unsieved positions is stabilised at the value below $\mathcal{L}_n^{\alpha,\kappa}=1/\left(1-\mathcal{C}_n^{\alpha,\kappa}\right)$. Let
\[
	z_n^\mathrm{L}=\argmin\limits_{i\in S_n,\,i<z}|z-i|,\qquad
	z_n^\mathrm{R}=\argmin\limits_{i\in S_n,\,i>z}|z-i|,\qquad
	\breve{S}_n =S_n\setminus\left\{z_n^\mathrm{L},z_n^\mathrm{R}\right\}
\]
and
\[
	I_n^\mathrm{L}=\left\{z\in\left\llbracket i,{z_n^\mathrm{L}-1}\right\rrbracket\::\:i\in\mathbb{Z},
	\sum\limits_{u=i}^{z_n^\mathrm{L}}\mathcal{P}_n^{\alpha,\kappa,\mathfrak{r^*}}(u)=0,\:
	\mathcal{P}_n^{\alpha,\kappa,\mathfrak{r^*}}(i-1)=1\right\},
\]
\[
	I_n^\mathrm{R}=\left\{z\in\left\llbracket {z_n^\mathrm{R}+1},i\right\rrbracket\::\:i\in\mathbb{Z},
	\sum\limits_{u=z_n^\mathrm{R}}^i\mathcal{P}_n^{\alpha,\kappa,\mathfrak{r^*}}(u)=0,\:
	\mathcal{P}_n^{\alpha,\kappa,\mathfrak{r^*}}(i+1)=1\right\}.
\]
The sieved positions $z_n^\mathrm{L}$ and $z_n^\mathrm{R}$, closest to $\mathcal{S}_n^{\alpha,\kappa,\mathfrak{r^*}}(z)$ on the left and right hand side in the $n$-th step (see the examples for $n$, $n+1$, $n+2$ in Fig.~\ref{fig:sieve}), append the new intervals $I_n^\mathrm{L}$ and $I_n^\mathrm{R}$ to the sieve, both having the average length $\widetilde{\mathcal{L}}_n^{\alpha,\kappa}(z)$, hence, the sieve expands with the average speed equal to $2\widetilde{\mathcal{L}}_n^{\alpha,\kappa}(z)$ per step. This process does not have to be perfectly regular, in the sense that $z_{n+1}^\mathrm{L}=\min\left(\mathcal{S}_n^{\alpha,\kappa,\mathfrak{r}^*}(z)\right)-1$ and $z_{n+1}^\mathrm{R}=\max\left(\mathcal{S}_n^{\alpha,\kappa,\mathfrak{r}^*}(z)\right)+1$ for each $n\in\mathbb{Z}^+$, because partial rearrangements of the sequences $\left(z_n^\mathrm{L}\right)$ and $\left(z_n^\mathrm{R}\right)$ are irrelevant to the effective sieve expansion. The remaining sieved positions inserted to $\mathcal{P}_n^{\alpha,\kappa,\mathfrak{r^*}}$ in the $n$-th step, i.e., those included in $\breve{S}_n$, increase the average density of sieved positions in $\mathcal{P}_n^{\alpha,\kappa,\mathfrak{r^*}}$ and, as a result, increase the average speed of further expansion of the sieve. However, these sieved positions cannot increase $\widetilde{\mathcal{L}}_n^{\alpha,\kappa}(z)$ over $\mathcal{L}_n^{\alpha,\kappa}$, because $\widetilde{\mathcal{L}}_n^{\alpha,\kappa}(z)$ remains effectively determined by the superposition~(\ref{eq:superposition}), which asserts that $\widetilde{\mathcal{L}}_n^{\alpha,\kappa}(z)<\mathcal{L}_n^{\alpha,\kappa}$, as we have shown before. In particular, according to Definition~\ref{def:regularPrimeSeq}, Lemma~\ref{lem:primesOmN}, and Lemma~\ref{lem:patAvgLenAsymp}, we have $\mathfrak{p}_n=\omega\left(\mathcal{L}_n^{\alpha,\kappa}\right)$, hence, the distribution of sieved positions virtually does not change in the neighbourhood of the positions where $z_n^\mathrm{L}$ and $z_n^\mathrm{R}$ are inserted (except for the insertions themselves) and we can correctly assume that the residues $z_n^\mathrm{L}$ and $z_n^\mathrm{R}$ are continuously inserted into a static pattern prepared in advance as a result of the mentioned superposition. As a consequence, the average speed of the expansion does not grow systematically faster than up to $2\mathcal{L}_n^{\alpha,\kappa}$ per step in the $n$-th step of the asymptotic behaviour. Therefore, we get that $\left(\#\mathcal{S}_n^{\alpha,\kappa,\mathfrak{r}^*}(z)\right)$ cannot grow regularly faster than the sequence
\[
	\beta_n^*=2\sum\limits_{i=1}^n\mathcal{L}_i^{\alpha,\kappa}
	=2\sum\limits_{i=1}^n\left(\mathcal{D}_i^{\alpha,\kappa}\right)^{-1},\qquad n\in\mathbb{Z}^+,
\]
hence, it oscillates infinitely many times around $\beta_n$, such that $\beta_n \leq \beta_n^*$. According to~(\ref{eq:patAvgDensReg}), $\mathcal{D}_{n+1}^{\alpha,\kappa}<\mathcal{D}_n^{\alpha,\kappa}$ for each $n\in\mathbb{Z}^+$, and we have
\begin{equation}
	\beta_n^*=2\sum\limits_{i=1}^n\left(\mathcal{D}_i^{\alpha,\kappa}\right)^{-1}
	\;\leq\;\frac{2n}{\mathcal{D}_n^{\alpha,\kappa}}=\gamma_n,\qquad n\in\mathbb{Z}^+.
	\label{eq:betaGamma}
\end{equation}
Applying~(\ref{eq:patAvgDensReg}) in~(\ref{eq:betaGamma}), we obtain~(\ref{eq:totalSieveUB}).
\end{proof}
\begin{figure}
	\centering
	\includegraphics[width=\textwidth]{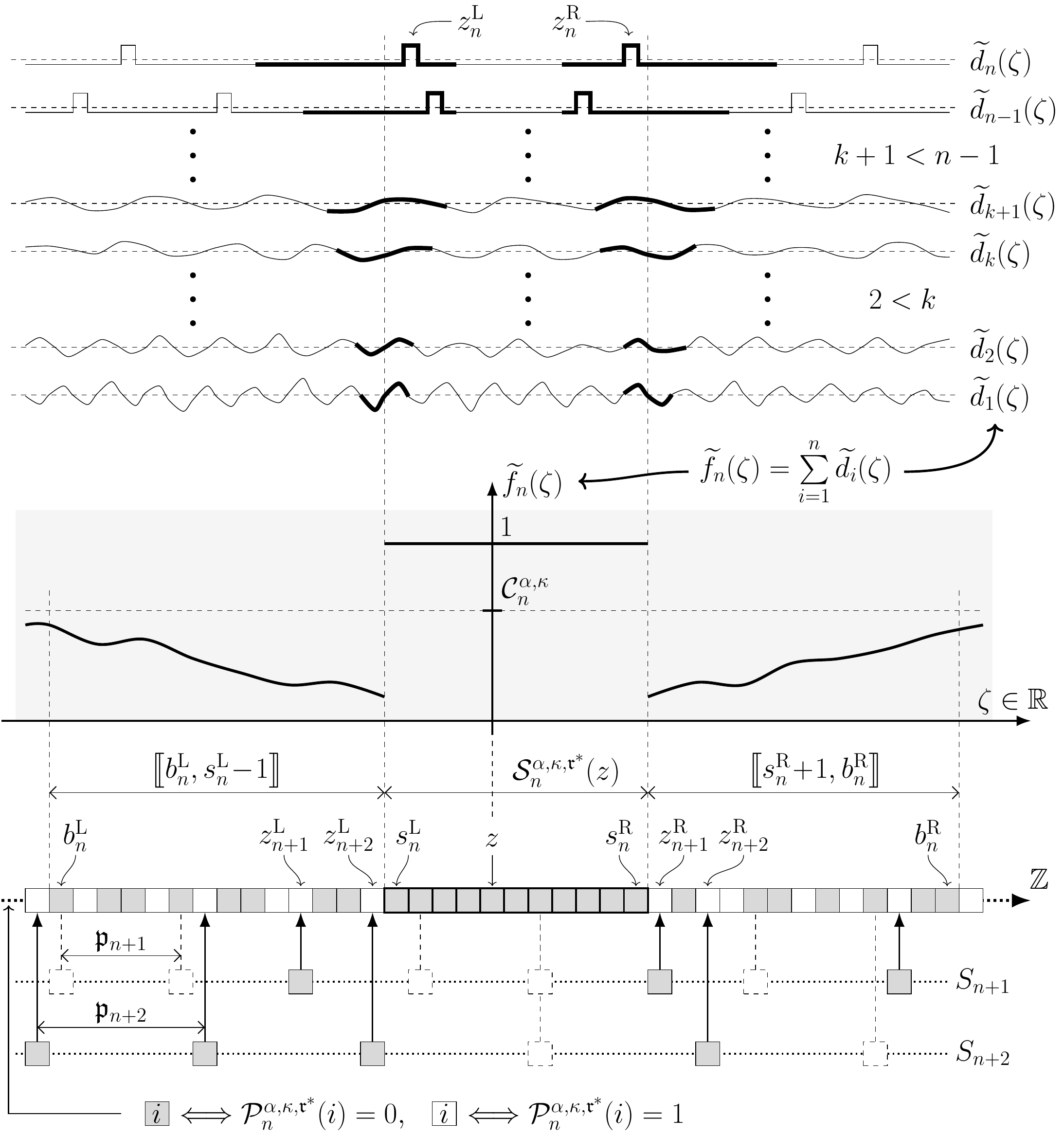}
	\caption{The model of an expanding total sieve and its surrounding}
	\label{fig:sieve}
\end{figure}

\begin{proposition}
For the sequence $\left(\gamma_n\right)$ used in the formula~(\ref{eq:totalSieveUB}) of Lemma~\ref{lem:totalSieveUB}, we have $\gamma_n=o(n^2)$.
\label{prop:primesOm}
\end{proposition}
\begin{proof}
According to the given proof of Lemma~\ref{lem:totalSieveUB} and according to~(\ref{eq:patAvgDensReg}), $\gamma_n=2n/\mathcal{D}_n^{\alpha,\kappa}$. From Lemma~\ref{lem:patAvgLenAsymp}, we obtain $\gamma_n=2n\,o(n)=o(n^2)$.
\end{proof}

\section{Admissible patterns in prime numbers}
\label{sec:patterns}
\begin{definition}
An admissible $k$-tuple $T=(a_1,a_2,\ldots,a_k)$ is a collection of integers $a_1<a_2<\ldots<a_k$, such that $\bigcup_{i=1}^k [a_i]_p\neq\mathbb{Z}$ for each $p\in\mathbb{P}$.
\end{definition}

\begin{definition}
Let $\mathcal{P}:\mathbb{Z}\mapsto\{0,1\}$. We will say that a $k$-tuple $T=(a_1,a_2,\ldots,a_k)$ matches the pattern $\mathcal{P}$ at a position $m\in\mathbb{Z}$ if $\mathcal{P}(m+a_i)=1$ for each $i\in\llbracket 1,k\rrbracket$.
\end{definition}

\begin{definition}
Let $T=(a_1,a_2,\ldots,a_k)$ be an admissible $k$-tuple. Let $d\in\mathbb{Z}^+\setminus\{1\}$, such that $p_d>a_k-a_1$. Let $m\in\mathbb{Z}^+$, $m<p_{d-1}\#$, be a position at which $T$ matches $\mathcal{P}_{d-1}$. Let $g\in\mathbb{Z}^+$. A function $^\dagger\mathcal{P}^{T,k,d,m}_g:\mathbb{Z}\mapsto\{0,1\}$, such that
\begin{equation}
	^\dagger\mathcal{P}^{T,k,d,m}_g(z)=\begin{cases}
		0 &\text{if }\big(\exists\,(u,v)\in \llbracket 1,k\rrbracket\times \llbracket d,d\!+\!g\!-\!1\rrbracket\big)\\
		& \qquad m+a_u+(z-1)p_{d-1}\#\equiv 0 \Mod{p_v},\\
		1 &\text{otherwise},
	\end{cases}
\label{eq:tuplePrimorial}
\end{equation}
will be called a tuple-primorial sieving pattern.
\label{def:tuplePrimorial}
\end{definition}

\begin{lemma}
If $z\in\mathbb{Z}$ and $^\dagger\mathcal{P}^{T,k,d,m}_g(z)=1$, then the $k$-tuple $T$ matches the Eratosthenes sieving pattern $\mathcal{P}_{d+g-1}$ at the position $m+(z-1)p_{d-1}\#$.
\label{lem:tuplePrimToPrimes}
\end{lemma}
\begin{proof}
By Definition~\ref{def:tuplePrimorial}, $T$ matches $\mathcal{P}_{d-1}$ at the position $m$. By Definition~\ref{def:patEratosthenes}, Lemma~\ref{lem:sievePatPeriod}, and according to the formula~(\ref{eq:patPeriodReg}), we obtain that $\mathcal{P}_{d-1}$ is a periodic function with the fundamental period $\mathcal{T}_{d-1}=p_{d-1}\#$. Hence, $T$ matches $\mathcal{P}_{d-1}$ at each position $m+i p_{d-1}\#$, $i\in\mathbb{Z}$. Consequently, $T$ matches $\mathcal{P}_{d+g-1}$ at the position $m+(z-1)p_{d-1}\#$, $z\in\mathbb{Z}$, if and only if $m+a_u+(z-1)p_{d-1}\#\not\equiv 0 \Mod{p_v}$ for each $(u,v)\in \llbracket 1,k\rrbracket\times \llbracket d,d\!+\!g\!-\!1\rrbracket$. The latter condition asserts that no element of the $k$-tuple $T$ matching $\mathcal{P}_{d-1}$ at the position $m+(z-1)p_{d-1}\#$ is sieved out by any $p_v\in\mathbb{P}$ for $v\in\llbracket d,d\!+\!g\!-\!1\rrbracket$. According to the formula~(\ref{eq:tuplePrimorial}) of Definition~\ref{def:tuplePrimorial}, this condition is satisfied if $^\dagger\mathcal{P}^{T,k,d,m}_g(z)=1$.
\end{proof}

\begin{lemma}
For each tuple-primorial sieving pattern $^\dagger\mathcal{P}^{T,k,d,m}_g$, there exists an $(\alpha,\kappa)$--regular sieving pattern $\mathcal{P}^{d,k,\mathfrak{r}}_{kg}$, such that $^\dagger\mathcal{P}_g^{T,k,d,m}=\mathcal{P}^{d,k,\mathfrak{r}}_{kg}$.
\label{lem:tupleSieveing}
\end{lemma}
\begin{proof}
Choose arbitrarily a tuple-primorial sieving pattern $^\dagger\mathcal{P}^{T,k,d,m}_g$ related to a $k$-tuple $T=(a_1,a_2,\ldots,a_k)$. Notice that, according to Definition~\ref{def:tuplePrimorial}, $^\dagger\mathcal{P}^{T,k,d,m}_g$ is the characteristic function of the set $\mathbb{Z}\setminus{^\dagger\mathcal{M}^{T,k,d,m}_g}$ such that, from~(\ref{eq:tuplePrimorial}), we have
\begin{equation}
	^\dagger\mathcal{M}^{T,k,d,m}_g=\bigcup\limits_{v=d}^{d+g-1}\bigcup\limits_{u=1}^k
	\Big(\big\{\left(xp_v-m-a_u\right)/p_{d-1}\#+1,\,x\in\mathbb{Z}\big\}\cap\mathbb{Z}\Big).
\label{eq:tupleSieveingArithmA}
\end{equation}

It is obvious that $(p_v,p_{d-1}\#)=1$ for any $v\in V=\llbracket d,d+g-1\rrbracket$, therefore, the formula~(\ref{eq:tupleSieveingArithmA}) can be simplified to
\begin{equation}
	^\dagger\mathcal{M}^{T,k,d,m}_g=\bigcup\limits_{v=d}^{d+g-1}\bigcup\limits_{u=1}^k
	\big\{xp_v+r_{u,v},\,x\in\mathbb{Z}\big\}
	=\bigcup\limits_{v=d}^{d+g-1}\bigcup\limits_{u=1}^k[r_{u,v}]_{p_v},
\label{eq:tupleSieveingArithmB}
\end{equation}
where $r_{u,v}\in\mathbb{Z}$ is a function of $p_v$, $a_u$, $m$, and $p_{d-1}$, which has the property such that for any $u_1,u_2\in\llbracket 1,k\rrbracket$, $v\in V$, if $u_1 \neq u_2$, then $r_{u_1,v} \not\equiv r_{u_2,v} \pmod{p_v}$. Indeed, assuming $r_{u_1,v} \equiv r_{u_2,v} \pmod{p_v}$, we obtain, by comparison of~(\ref{eq:tupleSieveingArithmA}) and~(\ref{eq:tupleSieveingArithmB}), that $a_{u_1} \equiv a_{u_2} \pmod{p_v}$, and consequently $u_1=u_2$, because $\left|a_{u_1}-a_{u_2}\right|\leq a_k-a_1<p_d\leq p_v$, according to Definition~\ref{def:tuplePrimorial}. Therefore, the set $^\dagger\mathcal{M}^{T,k,d,m}_g$ is a union of pairwise distinct residue classes which can be considered as the first $kg$ elements of a sequence $\left([\mathfrak{r}_i]_{\mathfrak{p}_i}\right)$ such that $\left(\mathfrak{r}_i\right)$ conforms the definition of residue sieving sequence (Definition~\ref{def:sievSequence}) and $\left(\mathfrak{p}_i\right)$ conforms the definition of $(\alpha,\kappa)$--regular prime sieving sequence (Definitions~\ref{def:sievSequence} and~\ref{def:regularPrimeSeq}) under $\alpha=d$ and $\kappa=k$. As a result, we have $^\dagger\mathcal{M}^{T,k,d,m}_g=\mathcal{M}_{kg}^{d,k,\mathfrak{r}}$ and, according to Definition~\ref{def:sievPattern} and Remark~\ref{rem:regPatDef}, $^\dagger\mathcal{P}^{T,k,d,m}_g=\mathcal{P}_{kg}^{d,k,\mathfrak{r}}$.
\end{proof}

\begin{mtheorem}
Every admissible $k$-tuple matches infinitely many positions in the sequence of prime numbers.
\end{mtheorem}
\begin{proof}
Choose any admissible $k$-tuple $T$ and suppose that it matches only finitely many positions in the sequence of primes. Consider a tuple-primorial sieving pattern $^\dagger\mathcal{P}^{T,k,d,m}_n$ with arbitrary values of $d$ and $m$, satisfying Definition~\ref{def:tuplePrimorial}. Let $z\in\mathbb{Z}$. According to Lemma~\ref{lem:tuplePrimToPrimes}, $^\dagger\mathcal{P}^{T,k,d,m}_n(z)=1$ implies that $T$ matches the Eratosthenes sieving pattern $\mathcal{P}_{d+n-1}$ at the position $m+(z-1)p_{d-1}\#$. By Lemma~\ref{lem:patEratPrimes}, if $z\in\left\llbracket 2,p_{n+1}^2\!-\!1\right\rrbracket$ and ${\mathcal{P}_n(z)=1}$, then $z\in\mathbb{P}$. Therefore, if $T$ matches primes only finitely many times, there exists ${z^*\in\mathbb{Z}^+}$, such that $^\dagger\mathcal{P}^{T,k,d,m}_n(z)=0$ for each position $z\in Z_n^{d,m}=\left\llbracket z^*,(p_{d+n}^2-m)/p_{d-1}\#\right\rrbracket$. From Lemma~\ref{lem:tupleSieveing}, we have that there exist a residue sieving sequence $\mathfrak{r}$ and a related sequence of $(\alpha,\kappa)$--regular sieving patterns $\left(\mathcal{P}^{d,k,\mathfrak{r}}_{kn}\right)$, such that $\mathcal{P}^{d,k,\mathfrak{r}}_{kn}={^\dagger\mathcal{P}^{T,k,d,m}_n}$ for any $n\in\mathbb{Z}^+$, hence, the patterns $\mathcal{P}^{d,k,\mathfrak{r}}_{kn}$ include total sieves $\mathcal{S}^{d,k,\mathfrak{r}}_{kn}(\widetilde{z})$ for some fixed $\widetilde{z}\in\mathbb{Z}$, such that ${Z_n^{d,m}\subseteq\mathcal{S}^{d,k,\mathfrak{r}}_{kn}(\widetilde{z})}$. It implies the following asymptotic relations
\begin{equation}
	\#Z_n^{d,m}\asymp p_{d+n}^2\sim p_n^2\ll \#\mathcal{S}^{d,k,\mathfrak{r}}_{kn}(\widetilde{z})
	\quad \text{under } n\to\infty.
\label{eq:sievingAsymp}
\end{equation}
From Lemma~\ref{lem:totalSieveUB} and Proposition~\ref{prop:primesOm}, we obtain that there exists a strictly increasing sequence $\left(\nu_n\right)$, where $\nu_n\in\mathbb{Z}^+$, such that $\#\mathcal{S}^{d,k,\mathfrak{r}}_{k\nu_n}(\widetilde{z})=o(\nu_n^2)$ under $n\to\infty$. Applying Lemma~\ref{lem:primesOmN}, we get $\#\mathcal{S}^{d,k,\mathfrak{r}}_{k\nu_n}(\widetilde{z})=o(p_{\nu_n}^2)$, which is in contradiction with~(\ref{eq:sievingAsymp}).
\end{proof}

\bibliographystyle{amsplain}
\bibliography{references}

\end{document}